\documentclass[a4paper, 12pt]{article}
\usepackage[utf8]{inputenc}
\usepackage{a4}
\usepackage[T1]{fontenc}
\usepackage[english]{babel}

\usepackage{amssymb,amsmath,amsthm}
\usepackage{enumerate}
\usepackage{graphicx}
\usepackage{epsfig}
\usepackage{comment}
\usepackage{xspace}
\usepackage{float}
\usepackage{multirow}
\usepackage{extarrows}
\usepackage{mathtools}
\usepackage{color}
\usepackage{nomencl}
\usepackage{nicefrac}
\usepackage{cite}
\makenomenclature
\newcommand{\n}{\ensuremath{\nabla}\xspace}
\newcommand{\R}{\ensuremath{\mathbb{R}}\xspace}

\newcommand{\eps}{\epsilon}
\newcommand{\loc}{\mathrm{loc}}

\renewcommand{\epsilon}{\varepsilon}

\newcommand{\palpha}{\partial_\alpha}
\newcommand{\udel}{u_\delta}

\newcommand{\Fdel}{F_\delta}

\newcommand{\phidel}{\varphi_\delta}

\newcommand{\dx}{\, \mathrm{d}x}

\newcommand{\intom}{\intop_{\Omega}}
\newcommand{\intoms}{\intop_{\Omega^*}}
\newcommand{\intomd}{\intop_{\Omega-D}}
\newcommand{\nabu}{\nabla u}

\newcommand{\nabdel}{\nabla u_\delta}

\newcommand{\gr}[1]{(\ref{#1})}

\newcommand{\leb}[1]{\mathcal{L}^{#1}}
\newcommand{\lamh}{\frac{\lambda}{2}}

\def\YYint#1#2#3{{\setbox0=\hbox{$#1{#2#3}{\int}$}
    \vcenter{\hbox{$#2#3$}}\kern-.5\wd0}}

\begin{document}

\numberwithin{equation}{section}
\newtheoremstyle{break}{15pt}{15pt}{\itshape}{}{\bfseries}{}{\newline}{}
\theoremstyle{break}
\newtheorem*{Satz*}{Theorem}
\newtheorem*{Rem*}{Remark}
\newtheorem*{Lem*}{Lemma}
\newtheorem{Satz}{Theorem}[section]
\newtheorem{Rem}{Remark}[section]
\newtheorem{lemma}{Lemma}[section]
\newtheorem{proposition}{Proposition}[section]
\newtheorem{Cor}{Corollary}[section]
\theoremstyle{definition}
\newtheorem{Def}[Satz]{Definition}

\parindent2ex

\newenvironment{rightcases}
  {\left.\begin{aligned}}
  {\end{aligned}\right\rbrace}

\begin{center}{\Large \bf On the solvability in Sobolev spaces and related regularity results for a variant of the TV-image recovery model: the vector-valued case}
\end{center}

\begin{center}
M. Bildhauer, M.Fuchs, J. M\"uller, C. Tietz
\end{center}



\begin{abstract}
We study classes of variational problems with energy densities of linear growth acting on vector-valued functions. 
Our energies are strictly convex variants of the TV-regularization model introduced by Rudin, Osher and Fatemi 
\cite{ROF} as a powerful tool in the field of image recovery. In contrast to our previous work we here try to figure 
out conditions under which we can solve these variational problems in classical spaces, e.g. in the Sobolev class $W^{1,1}$. 
\end{abstract}

\section{Introduction}\label{intro} 
In their fundamental paper \cite{ROF} Rudin, Osher and Fatemi proposed to study the variational problem (``TV-regularization``)
\begin{align}\label{1.1}
 I[u]:=\intom |\nabu|\dx+\lamh\intom|u-u_0|^2\dx\rightarrow\min
\end{align}
as a suitable model for the denoising of a grey-scale image $u_0:\Omega\rightarrow [0,1]$. 
Here $\Omega$ is a bounded domain in $\R^2$ and $\lambda>0$ denotes a given parameter. 
As a matter of fact problem \gr{1.1} in general admits no solution in the Sobolev class $W^{1,1}(\Omega)$ 
(see, e.g., \cite{Ad} for a definition of the spaces $W^{k,p}_\loc(\Omega,\R^N)$), and therefore one has to 
pass to the space $BV(\Omega)$ consisting of functions $u\in L^1(\Omega)$ with finite total variation 
(compare \cite{Giu} or \cite{AFP}). Further unpleasant features of problem \gr{1.1} are that the energy 
density $|\nabu|$ is neither differentiable nor strictly convex. So, from the analytical point of view, it seems reasonable 
to replace \gr{1.1}  by more regular problems being still of linear growth in $\nabu$ including even the case of vector-valued functions 
$u:\R^n\supset\Omega\rightarrow \R^N$ in more than two variables and combine the 
denoising procedure with simultaneous inpainting.
We wish to note that such a modification of \gr{1.1} is not only of theoretical interest, the practical importance 
is indicated in the paper \cite{BFW2}.

Next we fix our precise assumptions and state the main results: let $\Omega$ denote a bounded Lipschitz region in $\R^n$, $n\geq 2$ 
(the case $n=1$ is discussed in \cite{FMT}), 
and consider a $\mathcal{L}^n$-measurable subset $D$ of $\Omega$ such that 
\begin{align}\label{1.2}
 0\leq \mathcal{L}^n(D)<\mathcal{L}^n(\Omega).
\end{align}
The set $D$ represents the inpainting region, on which the data are missing, i.e.~in contrast to problem \gr{1.1} our noisy data $u_0:\Omega-D\rightarrow\R^N$ can just be observed on the region $\Omega-D$, and we require 
\begin{align}\label{1.3}
 u_0\in L^\infty(\Omega-D,\R^N).
\end{align}
For a fixed positive parameter $\lambda>0$ we then look at the variational problem
\begin{align}
 \label{1.4}
\begin{split}
 \left\{\begin{aligned}
  &J[u]:=\intom F(\nabu)\dx+\lamh\intop_{\Omega-D}|u-u_0|^2\dx\rightarrow \min\\
  &\text{among functions }u:\Omega\rightarrow\R^N,
 \end{aligned}\right.
\end{split}
\end{align}
where the choice $D=\emptyset$ corresponds to pure denoising. In the case $\leb{n}(D)>0$ the idea of \gr{1.4} is to denoise the incomplete data $u_0$ through the solution $u:\Omega\rightarrow\R^N$, which at the same time fills in the observed image on the missing region $D$. Concerning the density $F$ our assumptions are as follows: there are constants $\nu_i>0$ such that
\begin{eqnarray}
 \label{1.5} &F\in C^2(\R^{nN}),\;F\geq 0\text{ and (w.l.o.g.) }F(0)=0,&\\
 \label{1.6} &|DF(Z)|\leq \nu_1,&\\
 \label{1.7} &F(Z)\geq \nu_2|Z|-\nu_3,&\\
 \label{1.8} &\nu_4(1+|Z|)^{-\mu}|X|^2\leq D^2F(Z)(X,X)\leq \nu_5(1+|Z|)^{-1}|X|^2&
\end{eqnarray}
hold for all $X$, $Z\in\R^{nN}$. Condition \gr{1.8} is known as $\mu$-ellipticity, and we always require 
\begin{align}
 \label{1.9}\mu\in (1,\infty).
\end{align}
We note that clearly the second inequality in \gr{1.8} implies \gr{1.6}, and the reader, who is interested in minimal requirements concerning $F$ in specific situations, should consult the references given below. However, it turned out, that the hypotheses \gr{1.2}-\gr{1.9} are sufficient for proving the following results.

\noindent {\bf I. Existence and uniqueness.}
\begin{enumerate}[i)]
 \item The relaxed variant of problem \gr{1.4} defined on the space $BV(\Omega,\R^N)$ admits at least one solution $u$ being unique on $\Omega-D$.
 \item The absolutely continuous part $\nabla^au$ (with respect to Lebesgue's measure) of the matrix-valued measure $\nabla u$ is unique.
 \item Any minimizer of the relaxed problem occurs as a ($L^1$-) limit of a $J$-minimizing sequence from the space $W^{1,1}(\Omega,\R^N)$.
\end{enumerate}

We refer to the papers \cite{FT} and \cite{MT}, earlier contributions in more specific settings can be found for instance in \cite{BF1} and \cite{BF3}.

\noindent{\bf II. Duality.}

The problem being in duality to \gr{1.4} has a unique solution $\sigma\in L^\infty(\Omega,\R^{nN})$, and it holds $\sigma=DF(\nabla^au)$ a.e. on $\Omega$. 

The details can be found in \cite{FT} and \cite{BF3}.

\noindent{\bf III. Regularity.}
\begin{enumerate}[i)]
 \item Suppose that $N=1$ or assume 
\begin{align}
 \label{1.10} F(Z)=\Phi\big(|Z|\big),
\end{align}
 if the case $N>1$ is considered. In addition we require (compare \gr{1.9})
\begin{align}
 \label{1.11}\mu\in (1,2).
\end{align}
Then problem \gr{1.4} admits a unique solution $u$ in the space $W^{1,1}(\Omega,\R^N)$. This solution satisfies the "maximum principle"
\begin{align}
 \label{1.12}\sup_\Omega|u|\leq \sup_{\Omega-D}|u_0|.
\end{align}
\item Under the assumptions of i) $u$ is of class $C^{1,\alpha}(\Omega,\R^N)$ for any $\alpha\in (0,1)$.
\item If condition \gr{1.11} is dropped, then - keeping the hypothesis \gr{1.10} - we have partial $C^1$-regularity for any solution $u\in BV(\Omega,\R^N)$ of the relaxed variant of \gr{1.4}, i.e. there exists an open subset $\Omega_0$ of $\Omega$ such that $\leb{n}(\Omega-\Omega_0)=0$ and $u\in C^{1,\alpha}(\Omega_0,\R^N)$ for any $\alpha\in (0,\nicefrac{1}{2})$.
\end{enumerate}

A discussion of i) can be found in \cite{BF1} and \cite{BF2}, for a general proof of i) we refer to Section 3.4 of \cite{Ti}. 

In case $n=2$ ii) was established in \cite{BFT}, 
Section 3.5 of \cite{Ti} is devoted to the general case. 

Finally, statement iii) can be found in Section 3.3 of \cite{Ti}.
Originally it was proved in \cite{MT}, and the approach heavily benefits from the work \cite{Sc}. 
We wish to emphasize that in the vector-case $N>1$ the proof of the regularity results III depend on the structure condition 
\gr{1.10} in an essential way, since \gr{1.10} "always" implies inequality \gr{1.12}, which in turn gives the boundedness of the (relaxed) 
minimizer $u$ on account of \gr{1.3}. 

So the natural question occurs what kind of regularity results can be expected in the vector-case without 
imposing \eqref{1.10}. As a matter of fact, we can not hope for everywhere regularity in the sense of ii), but i) and iii) seem to be in reach. To 
begin with, we look at the case $n=2$.

\begin{Satz}\label{Thm1.1} Let $n=2$ and suppose that \gr{1.2}, \gr{1.3} as well as \gr{1.5}-\gr{1.9} are valid. Then, if
\begin{eqnarray*}
\mbox{either:}& \mu<2 & \mbox{together with $D=\emptyset$}\\
\mbox{or:}&  \mu < \frac{3}{2} & \mbox{in case of general $D$,}
\end{eqnarray*}
the following statements hold:
\begin{enumerate}[a)]
 \item Problem \gr{1.4} admits a unique solution $u$ in the space $W^{1,1}(\Omega,\R^N)$.
 \item The function $u$ is in any space $W^{1,p}_\loc(\Omega,\R^N)$, $p\in [1,\infty)$, in addition it holds
\begin{align*}
 u\in W^{2,s}_\loc(\Omega,\R^N),\;s<2.
\end{align*}
\item There is an open subset $\Omega_0$ of $\Omega$ such that $\mathcal{H}\text{-}\mathrm{dim}(\Omega-\Omega_0)=0$, 
i.e.~$\mathcal{H}^\eps(\Omega-\Omega_0)=0$ for all $\eps>0$, and $u\in C^{1,\alpha}(\Omega_0,\R^N)$ for any choice of $\alpha\in (0,1)$.
\end{enumerate}
\end{Satz}

\begin{Rem}\label{Rem1.1}
 Theorem \ref{Thm1.1} follows from Theorem 1.4 in \cite{FM} through simplification: we just let $m=1$ in this reference 
and observe that for $m=1$ the density result \cite{FM}, Theorem 1.1 holds automatically, 
if $n=2$ (compare Lemma 2.1, Lemma 2.2 and Remark 2.1, Remark 2.2 in \cite{FT}), 
which means that our hypotheses on $D$ are sufficient for proving Theorem \ref{Thm1.1}.
\end{Rem}

\begin{Rem}\label{Rem1.2}
 Of course partial regularity for $BV$-minimizers $u$ of the relaxed problem in the sense 
that $u\in C^{1,\alpha}(\Omega_0,\R^N)$ for an open set $\Omega_0$ with $\leb{2}(\Omega-\Omega_0)=0$ should hold for any value $\mu>1$. However, a proof of this statement would require an inspection of the arguments outlined in \cite{AG}, 
which means that the data term $\lamh\intomd |u-u_0|^2\dx$ has to be incorporated. 
Since we are interested in the Sobolev space solvability of problem \gr{1.4}, 
we have to impose the upper bounds $\mu<2$ and $\mu<\nicefrac{3}{2}$, respectively on the parameter $\mu$. 
Hence these bounds naturally occur in Theorem \ref{Thm1.1} c), and at the same time guarantee better estimates 
for the size of the singular set (compare the discussion of the size of $\Omega-\Omega_0$ in \cite{FM}).
\end{Rem}

In the higher-dimensional case of pure denoising we have the following version of Theorem \ref{Thm1.1} a), b):
\begin{Satz}\label{Thm1.2}
 Let $n\geq 3$ and assume the validity of \gr{1.3},\gr{1.5}-\gr{1.9} together with $D=\emptyset$. If we assume
\begin{align}\label{1.13}
 \mu<2,
\end{align}
then problem \gr{1.4} is uniquely solvable in the space $W^{1,1}(\Omega,\R^N)$. 
Moreover, the solution is of class $W^{1,2}_\loc(\Omega,\R^N)\cap W^{2,\frac{4}{2+\mu}}_\loc(\Omega,\R^N)$.
\end{Satz}

The proof of Theorem \ref{Thm1.2} can be traced e.g.~in \cite{BF1}, however we will sketch its main ideas in Section 2. 
If we consider the case $\leb{n}(D)>0$ together with $n\geq 3$, then we only succeeded in proving  
the existence of a solution  $u$ in the space $W^{1,1}(\Omega,\R^N)$, 
if the growth order $r$ of the data term is not too large and at the same time a bound of the form 
$\mu<\mu(n,r)$ holds for the value of $\mu$. 

To be precise we replace the functional $J$ from \gr{1.4} through the expression
\begin{align}
\label{1.15} K[u]:=\intom F(\nabu)\dx+\intomd \omega\big(|u-u_0|\big)\dx
\end{align}
with $\omega:[0,\infty)\rightarrow [0,\infty)$, $\omega(0)=0$, 
being strictly increasing and strictly convex. 

In order not to overload our survey we restrict ourselves to the 
linear growth case and fix the example
\begin{align}\label{1.16}
\omega(t):=\sqrt{\beta^2+t^2}-\beta,\;t\geq 0,
\end{align}
with a fixed parameter $\beta>0$. We leave it to the reader to discuss the case of growth rates $r>1$ and 
to figure out what kind of bounds $\mu<\mu(n,r)$ have to replace the condition \gr{1.17} below.

\begin{Satz}\label{Thm1.3}
Assume that $\leb{n}(D)>0$, let $n\geq 3$, and suppose that \gr{1.3},\gr{1.5}-\gr{1.9} are valid. 
Moreover, we define $K$ and $\omega$ as in \gr{1.15} and \gr{1.16}, respectively. Suppose further that
\begin{align}
\label{1.17} \mu<\frac{3n}{3n-2}
\end{align}
is satisfied. Then it holds:
\begin{enumerate}[a)]
\item The problem $K\rightarrow\min$ in $W^{1,1}(\Omega,\R^N)$ admits a unique solution $u$ for which we have 
\begin{align*}
u\in W^{1,p}_\loc(\Omega,\R^N)\cap W^{2,s}_\loc(\Omega,\R^N),\\
p=\left(1-\frac{\mu}{2}\right)\frac{2n}{n-2},\;s=\frac{(2-\mu)n}{n-\mu}.
\end{align*}
\item There is an open set $\Omega_0\subset\Omega$ such that $\leb{n}(\Omega-\Omega_0)=0$ and $u\in C^1(\Omega_0,\R^N)$.
\end{enumerate}
\end{Satz}
The proof of Theorem \ref{Thm1.3} is presented in Section 3.

\begin{Rem}
As a matter of fact Theorem \ref{Thm1.3} remains valid in the case $D=\emptyset$.
\end{Rem}

The results from Theorem \ref{Thm1.3} suffer from the fact that the admissible range for the parameter $\mu$ stated in
\gr{1.17} decreases, if the dimension $n$ of the domain $\Omega$ increases. At the same time, the dimensionless bound 
stated in \gr{1.13} can only be established for the case of pure denoising ($D=\emptyset$)
together with the particular (quadratic) data term occuring in the functional $J$ from \gr{1.4}. 

However, following ideas outlined in Chapter 4.2 of \cite{Bi}, we can state a
result on ``Sobolev space solvability'' of the problem $K \to \min$ with general data term $\omega$ covering even the
case of a non-empty inpainting region $D$.

\begin{Satz}
\label{Thm1.4}
Under the conditions \gr{1.2} and \gr{1.3} for $D$ and $u_0$, respectively, assume that the density
$F$ satisfies \gr{1.5}-\gr{1.8} with parameter 
\begin{equation}\label{1.17a}
\mu \in (1,2).
\end{equation}
For $\delta >0$ let $u_\delta$ denote the unique solution of
\begin{equation}\label{1.18a}
K_\delta[w] := \frac{\delta}{2} \int_\Omega |\nabla w|^2 \dx + K[w] \to \min \, \mbox{in $W^{1,2}(\Omega,\R^N)$}
\end{equation}
with $K$ defined in \gr{1.15}, the function $\omega$: $[0,\infty) \to [0,\infty)$ being strictly convex and strictly increasing with $\omega(0) = 0$.
Then, if for any subdomain $\Omega^* \Subset \Omega$ it holds
\begin{equation}\label{1.19a}
\sup_{\delta >0} \|u_\delta\|_{L^\infty(\Omega^*,\R^N)} \leq c(\Omega^*) < \infty,
\end{equation}
the problem
\begin{equation}\label{1.20a}
K\to \min \, \mbox{in $W^{1,1}(\Omega,\R^N)$}
\end{equation}
admits a unique solution $u$, which satisfies $u \in W^{1,p}_\loc(\Omega,\R^N)$ for any $p < 4-\mu$.
\end{Satz}

\begin{Rem}\label{remnachthm1.4}
The (quadratic) regularization introduced in \gr{1.18a} is a well - established tool in connection
with linear growth problems for the reason that first the functions $u_\delta$ are sufficiently regular in order
to carry out certain calculations (in particular to establish Caccioppoli's inequality) and second, the minimizers $u_\delta$
converge (in an appropriate sense) towards the solution of e.g.~\gr{1.20a} as $\delta \to 0$. For an overview
of the properties of the $u_\delta$ we refer to e.g.~\cite{Bi} and \cite{Ti}, some more specific details are 
presented in the opening lines of the subsequent sections.
\end{Rem}

\begin{Rem}\label{rem1.5}
\begin{enumerate}[i)]
\item With respect to the Sobolev space solvability of the problems \gr{1.4} and \gr{1.20a} the upper bound $\mu < 2$ 
for the ellipticity exponent $\mu$ seems to be optimal and we refer the interested reader to
Remark 1.4 in \cite{FMT} for a more detailed discussion of the critical value $\mu =2$ even in the case $n=1$.
\item Let us look at variational problems of minimal surface type
\[
\int_{\Omega} F(\nabla u) \dx \to \min\,\, \mbox{in}\,\, \Phi +\stackrel{\circ}{W}\!\! {}^{1,1}(\Omega,\R^N)
\]
(compare problem ({$ \mathcal P$}) on p.~97 in \cite{Bi}) and its relaxed variant (see problem ({$\mathcal P$}') on p.~99
in \cite{Bi}) with $F$ satisfying \gr{1.5}-\gr{1.8} and a sufficiently regular boundary datum $\Phi$.

As it is outlined in Theorems 4.14 and 4.16 of \cite{Bi}, now the choice $\mu =3$ seems to be critical and an inspection
of the proof of Theorem \ref{Thm1.4} (see Remark \ref{rem 4.1}) will show: if $\mu < 3$ and if
\gr{1.19a} holds, then we can find a solution $u$, which is of class
$W^{1,1}(\Omega,\R^N) \cap W^{1,p}_\loc(\Omega,\R^N)$ for any $p < 4-\mu$.

This solution is unique up to a constant. Note that this result is a slight improvement of Theorem 4.16 in \cite{Bi}, 
since it does not require the structure condition for $N > 1$.
\end{enumerate}
\end{Rem}

\section{The proof of the theorem \ref{Thm1.2}}\label{PT} 
Following Remark \ref{remnachthm1.4} we consider the $\delta$-regularization of problem \gr{1.4}, 
i.e. for $\delta>0$ we let $\udel\in W^{1,2}(\Omega,\R^N)$ denote the unique solution of
\begin{align}
\label{2.1} J_\delta[u]:=\frac{\delta}{2}\intom |\nabu|^2\dx+J[u]\rightarrow\min\text{ in }W^{1,2}(\Omega,\R^N),
\end{align}
whose properties are summarized e.g. in Lemma 3.2 of \cite{FM}, where for the first order case at hand some obvious modifications in the statements have to be carried out. We claim the validity of 
\begin{align}
\label{2.2}\phidel:=\big(1+|\nabu|\big)^{1-\frac{\mu}{2}}\in W^{1,2}_\loc(\Omega)
\end{align}
uniformly with respect to the parameter $\delta$. In order to justify \gr{2.2}, we observe that the $J_\delta$-minimality of $\udel$ implies (recall \gr{2.1} and set $F_\delta:=F+\frac{\delta}{2}|\cdot|^2$)
\begin{align}
\label{2.3} \intom D^2\Fdel(\nabdel)(\palpha\nabdel,\nabla v)\dx=\lambda \intop_{\Omega-D}(u-u_0)\cdot\palpha v\dx,
\end{align}
$\alpha=1,...,n$, for any $v\in W^{1,2}(\Omega,\R^N)$ with compact support in $\Omega$.
Letting $v:=\eta^2\palpha\udel$ for some function $\eta\in C_0^1(\Omega)$, $0\leq \eta\leq 1$, we obtain (dropping the index $\delta$ and using summation w.r.t. $\alpha$) by inserting $v$ into equation \gr{2.3}
\begin{align}
\label{2.4}
\begin{split}
&\intom\eta^2D^2F(\nabu)(\palpha\nabu,\palpha\nabu)\dx+\intom 2D^2F(\nabu)(\eta\palpha\nabu,\nabla\eta\otimes\palpha u)\dx\\
&=\lambda\intomd (u-u_0)\cdot\palpha(\eta^2\palpha u)\dx .
\end{split}
\end{align}
If we apply the Cauchy-Schwarz inequality 
(for the bilinear form $D^2F(\nabu)=D^2\Fdel(\nabdel)$) and Young's inequality to the second term on the left-hand 
side of \gr{2.4}, it follows
\begin{align*}
\intom\eta^2D^2F(\nabu)(\palpha\nabu,\palpha\nabu)\dx\leq&2\intom D^2F(\nabu)(\nabla\eta\otimes\palpha u,\nabla\eta\otimes\palpha u)\dx\\
&+2\lambda\underset{\mbox{$=:T$}}{\underbrace{\intomd (u-u_0)\cdot\palpha(\eta^2\palpha u)\dx}},
\end{align*}
and the second part of \gr{1.8} yields on account of 
\begin{align*}
\sup_{\delta>0}\intom |\nabdel|\dx<\infty
\end{align*}
the bound
\begin{align}
\label{2.5}\intom\eta^2D^2F(\nabu)(\palpha\nabu,\palpha\nabu)\dx\leq c(\eta)+2\lambda T.
\end{align}
We discuss the quantity $T$. Performing an integration by parts we get
\begin{align*}
T=-\intom\eta |\nabu|^2\dx-\underset{\mbox{$=:T_1$}}{\underbrace{\intom u_0\palpha(\eta^2\palpha u)\dx}}
\end{align*}
and the boundedness of $u_0$ implies 
\begin{align*}
|T_1|&\leq c\intom |\nabla\eta||\nabu|+\eta^2|\nabla^2u|\dx\\
&\leq c(\eta)+c\intom\eta^2|\nabla^2u|\dx\\
&\leq c(\eta)+c\intom\eta^2(1+|\nabu|)^{-\frac{\mu}{2}}|\nabla^2u|(1+|\nabu|)^{\frac{\mu}{2}}\dx\\
&\underset{\mbox{\gr{1.8}}}{\leq}c(\eta)+\eps\intom D^2F(\nabu)(\palpha\nabu,\palpha\nabu)\eta^2\dx+c(\eps)\intom\eta^2|\nabu|^\mu\dx.
\end{align*}
Going back to \gr{2.5} and choosing $\eps$ in an appropriate way, we find
\begin{align*}
&\intom\eta^2D^2F(\nabu)(\palpha\nabu,\palpha\nabu)\dx+\intom\eta^2|\nabu|^2\dx\\
&\leq c(\eta)+c\intom \eta^2|\nabu|^\mu\dx.
\end{align*}
Since we assume $\mu<2$ (recall \gr{1.13}), we end up with (introducing the parameter $\delta$ again)
\begin{align}
\label{2.6}
\intop_{\Omega^*}D^2\Fdel(\nabdel)(\palpha\nabdel\palpha\nabdel)\dx+\intop_{\Omega^*}|\nabdel|^2\dx\leq c(\Omega^*)
\end{align}
for any subdomain $\Omega^*\Subset\Omega$ with a finite constant $c(\Omega^*)$ independent of $\delta$. 
Clearly \gr{2.6} implies \gr{2.2}, and since according to \gr{2.6} we have a local uniform bound for 
$\|\udel\|_{W^{1,2}_\loc(\Omega,\R^N)}$, the first claim of Theorem \ref{Thm1.2} follows along the lines of the proof of Theorem 1.3 in \cite{BF1}, 
moreover, we can choose $p=2$. If $s\in\left(1,\frac{4}{2+\mu}\right]$ is given, we obtain from H\"older's inequality
\begin{align*}
\intoms|\nabla^2\udel|^s\dx=\intoms\big(1+|\nabdel|\big)^{-\mu\frac{s}{2}}|\nabla^2\udel|^s\big(1+|\nabdel|\big)^{\mu\frac{s}{2}}\dx\\
\leq \left(\intoms\big(1+|\nabdel|\big)^{-\mu}|\nabla^2\udel|^2\dx\right)^\frac{s}{2}\left(\intoms\big(1+|\nabdel|\big)^\frac{\mu s}{2-s}\dx\right)^{1-\frac{s}{2}},
\end{align*}
and since $\nicefrac{\mu s}{2-s}\leq 2$ by our choice of $s$, we deduce from \gr{2.6} that we can select $s=\nicefrac{4}{2+\mu}$, which completes the proof of Theorem \ref{Thm1.2}.\qed
\begin{Rem}\label{Rem2.1} Sobolev's inequality combined with \gr{2.2} yields
\begin{align*}
\big(1+|\nabdel|\big)^{(2-\mu)\frac{n}{n-2}}\in L^1_\loc(\Omega)
\end{align*}
with exponent $(2-\mu)\frac{n}{n-2}>1$ iff $\mu$ satisfies $\mu<\frac{n+2}{n}$. 
Moreover, $(2-\mu)\frac{n}{n-2}\geq 2$ is equivalent to the unnatural requirement $\mu\leq\frac{4}{n}$ even contradicting \gr{1.9} 
in case $n\geq 4$. Thus we can not improve the above bound on the integrability exponent $s$ for $\nabla^2\udel$ by merely exploiting \gr{2.2}.
\end{Rem}
\section{The proof of the theorem \ref{Thm1.3}}\label{PT2} 
For the reader's convenience we summarize some properties of the solutions $u_\delta$ of problem \gr{1.18a} with $K$
and $\omega$ as defined in \gr{1.15} and \gr{1.16}, respectively.

\begin{lemma}\label{udproperties}
\begin{enumerate}[i)]
\item $\sup_\delta \|u_\delta\|_{W^{1,1}(\Omega, \R^N)} < \infty .$
\item $(u_\delta)_\delta$ is a $K$-minimizing sequence.
\item $u_\delta \in W^{2,2}_\loc(\Omega, \R^N)$.
\end{enumerate}
\end{lemma}

\begin{proof}[Proof of Lemma \ref{udproperties}]
The claim i) is obvious while the corresponding reference for a proof of part ii) is quoted in 
Remark \ref{remnachthm1.4}, i). 
Finally, the last statement iii) follows by using the well-known difference quotient technique.
\end{proof}

Let us now come to the proof of Theorem \ref{Thm1.3}.

\noindent{Ad a).} In what follows, we again suppress the index $\delta$ for notational simplicity and emphasize, that $u=u_\delta$ as well as $F=F_\delta$. For the proof of part a), we go back to inequality \gr{2.5}, in which the quantity $T$ is replaced by 
\begin{align*}
 T:=\intomd\underset{\mbox{$=:\xi$}}{\underbrace{\omega'\big(|u-u_0|\big)\frac{u-u_0}{|u-u_0|}}}\cdot\palpha(\eta^2\palpha u)\dx
\end{align*}
with $\xi\in L^\infty(\Omega)$ due to \gr{1.16}. Note, that by Young's inequality we have
\begin{align*}
 |T|\leq c\intom |\palpha(\eta^2\palpha u)|\dx\overset{\mbox{\gr{1.8}}}{\leq}\eps\intom D^2F(\nabu)(\palpha\nabu,\palpha\nabu)\eta^2\dx\\
 +c(\eps)\intom\eta^2\big(1+|\nabu|\big)^\mu\dx+\underset{\mbox{$\leq c(\eta)$}}{\underbrace{c\intom |\nabla\eta||\nabu|\dx}}.
\end{align*}
The first term in the above estimate can be absorbed in the left-hand side of \gr{2.5}, and we obtain
\begin{align}\label{3.1}
 \intom D^2F(\nabu)(\palpha\nabu,\palpha\nabu)\eta^2\dx\leq c\intom\eta^2\big(1+|\nabu|\big)^\mu\dx+c(\eta).
\end{align}
We let
\begin{align*}
 \varphi:=\big(1+|\nabu|\big)^{1-\frac{\mu}{2}},\quad \psi:=\big(1+|\nabu|\big)^{\mu(1-\frac{1}{n})}
\end{align*}
and recall
\begin{align}
 \label{3.2}|\nabla\varphi|^2\leq cD^2F(\nabu)(\palpha\nabu,\palpha\nabu).       
\end{align}
Moreover, we observe the identities (recall $0 \leq \eta \leq 1$)
\begin{eqnarray*}
 &\displaystyle \psi^\frac{n}{n-1} =\big(1+|\nabu|\big)^\mu,&\\
 &\displaystyle \eta^2\big(1+|\nabu|\big)^\mu = \eta^2 \psi^{\frac{n}{n-1}} \leq\eta^{\frac{n}{n-1}}\psi^\frac{n}{n-1}.&
\end{eqnarray*}
Now from Sobolev's and Poincaré's inequality it follows
\begin{align*}
\intom\eta^2\big(1+|\nabu|\big)^\mu\dx
\leq\|\eta\psi\|^{\frac{n}{n-1}}_{L^{\nicefrac{n}{n-1}}(\Omega)}\leq c\|\nabla(\eta\psi)\|_{L^1(\Omega)}^\frac{n}{n-1}\\
\leq\bigg(\underbrace{\int\limits_{\Omega}{|\nabla\eta\psi|\dx}}_{\mbox{$=:S_1$}}\bigg)^{\frac{n}{n-1}}+\bigg(\underbrace{\int\limits_{\Omega}{\eta|\nabla\psi|\dx}}_{\mbox{$=:S_2$}}\bigg)^{\frac{n}{n-1}}.
\end{align*}
In order to proceed further, we observe that \eqref{1.17} implies
\begin{align} 
\label{3.3} \mu\left(1-\frac{1}{n}\right)\leq 1\quad\Leftrightarrow\quad \mu\leq \frac{n}{n-1}.
\end{align}
From \eqref{3.3} it follows $\psi\leq(1+|\nabu|)$, hence
\begin{align}
 \label{3.4}S_1\leq c(\eta).
\end{align}
For handling the quantity $S_2$, we write
\begin{align*}
 \psi=\varphi^{\frac{2}{2-\mu}\mu\frac{n-1}{n}}
\end{align*}
(with exponent $\frac{2}{2-\mu}\mu\frac{n-1}{n}>1$ according to $\mu>1$). H\"older's and Young's inequality then yield
\begin{align*}
 S_2^\frac{n}{n-1}&=\left(\intom\eta |\nabla\psi|\dx\right)^\frac{n}{n-1}\\
&\leq c\left(\intom\eta\varphi^{\frac{2\mu}{2-\mu}\frac{n-1}{n}-1}|\nabla\varphi|\dx\right)^\frac{n}{n-1}\\
&\leq c\left(\intom\eta^2|\nabla\varphi|^2\dx\right)^{\frac{1}{2}\frac{n}{n-1}}\left(\intom \varphi^{\frac{4\mu}{2-\mu}\frac{n-1}{n}-2}\dx\right)^{\frac{1}{2}\frac{n}{n-1}}\\
&\leq\eps\intom\eta^2|\nabla\varphi|^2\dx+c(\eps)\left(\intop_{\text{spt}(\eta)}\varphi^{\frac{4\mu}{2-\mu}\frac{n-1}{n}-2}\dx\right)^\frac{n}{n-2}. 
\end{align*}
Finally we observe
\begin{align*}
 \varphi^{\frac{4\mu}{2-\mu}\frac{n-1}{n}-2}\leq c (1+|\nabla u|)
\end{align*}
which is a consequence of the inequality
\begin{align*}
 \left(\frac{4\mu}{2-\mu}\frac{n-1}{n}-2\right)\left(1-\frac{\mu}{2}\right)\leq 1
\end{align*}
being equivalent to \gr{1.17}. Now, by \gr{3.4} along with the above discussion of the quantity $S_2$ (and reintroducing the parameter $\delta$) we deduce from \gr{3.1}:
\begin{align}
 \label{3.5}\intop_{\Omega^*}D^2\Fdel(\nabdel)(\palpha\nabdel,\palpha\nabdel)\dx\leq c(\Omega^*)
\end{align}
uniform in $\delta$ for all compact subsets $\Omega^*\Subset\Omega$, i.e. (compare \gr{3.2})
\begin{align*}
 \varphi_\delta\in W^{1,2}_\loc(\Omega)\quad\text{and therefore}\quad |\nabu|\in L^p_\loc(\Omega)
\end{align*}
with $p=\left(1-\nicefrac{\mu}{2}\right)\left(\nicefrac{2n}{(n-2)}\right)$. Combining this information 
with the arguments used at the end of the proof of Theorem \ref{Thm1.2} we obtain
\begin{align*}
 |\nabla^2u_\delta|\in L^s_\loc(\Omega),
\end{align*}
where $s=\frac{(2-\mu)n}{n-\mu}$, and part a) of Theorem \ref{Thm1.3} is proved.

\noindent{Ad b).} Here we benefit from Corollary 3.3 in \cite{Sc} choosing $p=2$ 
in this reference. In order to justify the application of this corollary in our setting we first formulate a 
proposition which shows that we are actually in the situation of \cite{Sc}.
\begin{proposition}\label{Proppartial}
Under the hypotheses of Theorem \ref{Thm1.3} we have:
\begin{enumerate}[a)]
 \item for all $P\in\R^{nN}$ the density $F$ satisfies the hypotheses (H1)-(H4) of \cite{Sc} (see Section 2 in this reference);
\item setting $g:\Omega\times\R^{N}\rightarrow\R,\,\, g(x,y):= \chi_{\Omega-D}\omega(|y-u_0(x)|)$ the following statements hold true:
\begin{enumerate}[i)]
 \item $g$ is a Borel function;
\item there is a constant $C>0$ such that we have
\begin{align}
 \label{hc} |g(x,y_1)-g(x,y_2)|\leq C|y_2-y_1|
\end{align}
for all $x\in\Omega, y_1,y_2\in\R^N$.
\end{enumerate}
\end{enumerate}
\end{proposition}
\begin{proof}[Proof of Proposition \ref{Proppartial}]
In accordance with the hypotheses of Theorem \ref{Thm1.3} we can 
state that the density $F$ satisfies \eqref{1.5}-\eqref{1.8} for
our fixed $\mu<\frac{3n}{3n-2}$. 
For proving assertion a) we note that on account of \eqref{1.8} we deal with the non-degenerate case. Quoting Remark 2.6 in \cite{Sc} we then choose $p=2$ in this reference and as a consequence (H2), (H3) as well as (H4) in \cite{Sc} 
correspond to the requirement that $F$ is of class $C^2(\R^{nN})$ with $D^2F(P)>0$ for all $P\in\R^{nN}$. 
Thus, $F$ satisfies (H2)-(H4) by recalling \eqref{1.8}. 
Furthermore, $F$ fulfills (H1) since $F$ is (strictly) convex on $\R^{nN}$ (see \eqref{1.8} again) and of linear growth.
In order to verify the statements of part b) we first remark that assertion b), i) is immediate whereas a calculation of $\n_yg(x,y)$ directly 
gives \eqref{hc} (recall that the function $\omega$ is defined as in \eqref{1.16}).
\end{proof}

Using Corollary 3.3 in \cite{Sc} we may immediately conclude that there exists an open subset $\Omega_0$ of $\Omega$ such that 
$u\in C^{1}(\Omega_0,\R^N)$ together with 
$\mathcal{L}^n(\Omega-\Omega_0)=0$. This completes the proof of Theorem \ref{Thm1.3}.\qed

\section{The proof of the theorem \ref{Thm1.4}}\label{PT3}
 W.l.o.g.~we replace \gr{1.19a} by the global bound
\begin{equation}\label{4.1}
\sup_{\delta >0} \|u_\delta\|_{L^{\infty}(\Omega,\R^N)} < \infty
\end{equation}
and as usual drop the index $\delta$. From \gr{1.18a} we infer
\begin{equation}\label{4.2}
\int_{\Omega} DF(\nabla u):\nabla \Big(\eta^2 u \Gamma^{\frac{\alpha}{2}}\Big) \dx = 
\int_{\Omega} \Theta \cdot u \eta^2 \Gamma^{\frac{\alpha}{2}} \dx
\end{equation}
for a suitable function $\Theta \in L^{\infty}(\Omega,\R^N)$ (uniform in $\delta$).

Here we have abbreviated $\Gamma := 1+|\nabla u|^2$, $\alpha$ denotes a positive parameter and $\eta$ is a function in
$C^\infty(B_{2R}(x_0))$, $B_{2R}(x_0) \Subset \Omega$, such that $\eta =1$ on $B_R(x_0)$, $0\leq \eta \leq 1$ and
$|\nabla \eta| \leq c/R$. From the convexity of $F$ (see \gr{1.8}) together with \gr{1.7} and $F(0) =0$ it follows
\[
F(\nabla u) :\nabla u \geq c |\nabla u| ,
\]
hence \gr{4.2} yields on account of \gr{4.1}
\begin{eqnarray*}
\int_{\Omega} \eta^2 \Gamma^{\frac{\alpha +1}{2}} \dx &\leq &
c \Bigg[ \int_{\Omega} \eta^2 \Gamma^{\frac{\alpha}{2}} \dx +
\int_{\Omega} |DF(\nabla u)| |\nabla \eta^2| |u| \Gamma^{\frac{\alpha}{2}} \dx\\
&& + \int_\Omega |DF(\nabla u)|\eta^2 |u| \big| \nabla \Gamma^{\frac{\alpha}{2}}\big|\dx\Bigg]\\
&\leq & c \Bigg[ \int_{\Omega} \eta^2 \Gamma^{\frac{\alpha}{2}} \dx + 
\int_{\Omega} \eta |\nabla \eta| \Gamma^{\frac{\alpha}{2}} \dx
+ \int_{\Omega} \eta^2 |\nabla^2 u| \Gamma^{\frac{\alpha-1}{2}} \dx \Bigg] .
\end{eqnarray*}
Writing $\Gamma^{\alpha/2} = \Gamma^{(\alpha -1)/4} \Gamma^{(\alpha +1)/4}$ and applying
Young\rq{}s inequality, we find
\begin{eqnarray}\label{4.3}
\int_\Omega \eta^2 \Gamma^{\frac{\alpha +1}{2}}\dx & \leq &
c \Bigg[ \int_\Omega \eta^2 \Gamma^{\frac{\alpha -1}{2}} \dx 
+\int_{\Omega} |\nabla \eta|^2 \Gamma^{\frac{\alpha-1}{2}} \dx\nonumber\\
&& + \underbrace{\int_\Omega \eta^2 |\nabla^2 u| \Gamma^{\frac{\alpha -1}{2}} \dx}_{\mbox{$=:T$}}\Bigg] .
\end{eqnarray}
For discussing the integral $T$ we observe (compare \gr{4.2})
\[
\int_\Omega D^2F(\nabla u)(\nabla \partial_i u, \nabla (\eta^2 \partial_i u)) \dx 
= - \int_\Omega \Theta \partial_i (\eta^2 \partial_i u) \dx
\]
(summation w.r.t.~$i=1$, \dots, $n$), hence
\begin{eqnarray*}
\lefteqn{\int_\Omega D^2F(\nabla u) (\nabla \partial_i u, \nabla \partial_i u)\eta^2 \dx}\\
&=& - 2 \int_\Omega D^2F(\nabla u) (\nabla \partial_i u \eta, \nabla \eta \otimes \partial_i u)\dx
- \int_\Omega \Theta \partial_i (\eta^2 \partial_i u) \dx .
\end{eqnarray*}
Proceeding as done after \gr{2.4} (quoting \gr{1.8}) it follows as usual
\begin{eqnarray*}
\int_{\Omega} \Gamma^{-\frac{\mu}{2}} |\nabla^2 u|^2 \eta^2 \dx &\leq &
c \Bigg[ \int_\Omega |\nabla \eta|^2 |\nabla u| \dx 
+\int_\Omega |\nabla \eta| |\nabla u| \dx\\
&& + \int_\Omega \eta^2 |\nabla^2 u|\dx \Bigg]\\
&\leq &c(\eta) + c \int_\Omega \eta^2 |\nabla^2 u|\dx ,
\end{eqnarray*}
and a proper application of Young's inequality to the last term on the r.h.s.~yields
\begin{equation}\label{4.4}
\int_\Omega \eta^2 \Gamma^{-\frac{\mu}{2}} |\nabla^2 u|^2 \dx
\leq c(\eta) + c \int_\Omega \eta^2 \Gamma^{\frac{\mu}{2}} \dx ,
\end{equation}
hence (using \gr{4.4})
\begin{eqnarray*}
T&=& \int_\Omega |\nabla^2 u| \Gamma^{-\frac{\mu}{4}} \eta \eta \Gamma^{\frac{\mu}{4} + \frac{\alpha -1}{2}}\dx\\
&\leq & c \Bigg[ \int_\Omega \eta^2 |\nabla^2 u|^2 \Gamma^{-\frac{\mu}{2}} \dx
+ \int_\Omega \eta^2 \Gamma^{\alpha -1 + \frac{\mu}{2}} \dx \Bigg]\\
&\leq & c(\eta) + c \Bigg[ \int_\Omega \eta^2 \Gamma^{\frac{\mu}{2}}\dx
+\int_\Omega \eta^2 \Gamma^{\alpha -1 + \frac{\mu}{2}}\dx \Bigg] .
\end{eqnarray*}
Inserting this inequality into \gr{4.3} we find
\begin{eqnarray}\label{4.5}
\int_\Omega \eta^2 \Gamma^{\frac{\alpha + 1}{2}} \dx
&\leq & c(\eta)  \int_\Omega \Gamma^{\frac{\alpha-1}{2}} \dx\nonumber\\
&&+ c \Bigg[\int_\Omega \eta^2 \Gamma^{\frac{\mu}{2}} \dx
+ \int_{\Omega} \eta^2 \Gamma^{\alpha -1 + \frac{\mu}{2}}\dx \Bigg].
\end{eqnarray}
Let us assume
\begin{equation}\label{4.6}
\alpha \leq 2 .
\end{equation}
Then
\begin{equation}\label{4.7}
\int_\Omega \Gamma^{\frac{\alpha -1}{2}} \dx \leq c \Bigg[1+\int_{\Omega} |\nabla u| \dx\Bigg] \leq c < \infty
\end{equation}
for a constant $c$ independent of $\delta$, and if we further require
\begin{eqnarray}\label{4.8}
&\displaystyle \alpha +1 > \mu,&\\
\label{4.9}
&\displaystyle \alpha +1 > 2\alpha -2 + \mu \Leftrightarrow 3 > \alpha + \mu &
\end{eqnarray}
then Young's inequality applied to the last two terms on the
r.h.s.~of \gr{4.5} yields for subdomains $\Omega^* \Subset \Omega$
\begin{equation}\label{4.10}
\int_{\Omega^*} |\nabla u|^{\alpha +1} \dx \leq c(\Omega^*) < \infty 
\end{equation}
uniform in $\delta$. Recalling our assumption \gr{1.17a}, i.e.~$\mu \in (1,2)$,
we see that $\alpha < 3-\mu$ can be chosen arbitrarily close to the number
$3-\mu > 1$ satisfying the requirements \gr{4.6}, \gr{4.8}, \gr{4.9}, and \gr{4.10} reads (after introducing
the parameter $\delta$ again) as
\begin{equation}\label{4.11}
\sup_{\delta > 0} \| u_\delta\|_{W^{1,p}(\Omega^*,\R^n)} \leq c(\Omega^*,p) < \infty
\end{equation}
for any $\Omega^* \Subset \Omega$ and all $p < 4-\mu$.

Let us fix such a $p \in (1,4-\mu)$. From
\[
\sup_{\delta >0} \|u_\delta\|_{W^{1,1}(\Omega,\R^N)} < \infty
\]
we deduce the existence of $\bar{u} \in BV(\Omega,\R^N)$ such that, e.g., 
$u_\delta \to \bar{u}$ in $L^1(\Omega,\R^N)$ and a.e.~for a subsequence.
Moreover, by \gr{4.11}, it holds $\bar{u} \in W^{1,p}_\loc(\Omega,\R^N)$ and
therefore $\bar{u}\in W^{1,1}(\Omega,\R^N)$.

The lower semicontinuity of the functional $K$ (more precisely of its relaxed variant)
implies
\[
K[\bar{u}] \leq \liminf_{\delta \to 0} K[u_\delta] 
\]
(compare, e.g., the proof of Theorem 1.1 in \cite{FT} and use the fact that $\bar{u} \in W^{1,1}(\Omega,\R^N)$),
moreover, the minimality of $u_\delta$ shows (as $\delta \to 0$)
\[
K[u_\delta] \leq K_\delta[u_\delta] \leq K_\delta[v] \to K[v]
\]
for any $v \in W^{1,2}(\Omega,\R^N)$, thus
\[
K[\bar{u}] \leq K[v]
\]
for $v$ as above.

Quoting Lemma 2.1 from \cite{FT} we end up with
\[
K[\bar{u}] \leq K[w]\,\,\mbox{for all}\, w\in W^{1,1}(\Omega,\R^N),
\]
hence $\bar{u}$ is the (unique) solution of problem \gr{1.20a} in the Sobolev space
$W^{1,1}(\Omega,\R^N)$ satisfying in addition
$\nabla \bar{u} \in L^p_\loc(\Omega,\R^{nN})$ for $p < 4-\mu$.
\qed

\begin{Rem}\label{rem 4.1}
Let us look at the minimal surface case described in Remark \ref{rem1.5} $ii)$, where we
can choose $\Theta =0$ in \gr{4.2}. In place of \gr{4.5} we obtain
\[
\int_\Omega \eta^2 \Gamma^{\frac{\alpha +1}{2}}\dx
\leq c(\eta) \int_\Omega |\nabla \eta|^2 \Gamma^{\frac{\alpha -1}{2}}\dx
+ c \int_{\Omega} \eta^2 \Gamma^{\alpha -1 + \frac{\mu}{2}} \dx ,
\]
which makes the condition \gr{4.8} superfluous, and the requirement \gr{4.9}
can be satisfied at least for some $\alpha >0$,
provided we impose the bound $\mu < 3$ on the ellipticity parameter $\mu$.
\end{Rem}

\end{document}